\documentclass[12pt]{amsart}

\usepackage{amssymb,amsmath,amsthm}

\usepackage{graphicx, color}

\newtheorem{theorem}{Theorem}[section]
\newtheorem{lemma}[theorem]{Lemma}
\newtheorem{prop}[theorem]{Proposition}

\def \R{\mathbb{R}}

\def \grad{\nabla}

\def \half{\frac{1}{2}}

\numberwithin{equation}{section}

\begin{document}

\title[Transport MFAO]{A Transport Model for Multi-Frequency Acousto-Optic Tomography}

\author[Chung]{Francis J. Chung}
\address{Department of Mathematics, University of Kentucky, Lexington, KY, USA}
\email{fj.chung@uky.edu}

\author[Hoskins]{Jeremy G. Hoskins}
\address{Department of Mathematics, Yale University, New Haven, CT, USA}
\email{jeremy.hoskins@yale.edu}

\author[Schotland]{John C. Schotland}
\address{Department of Mathematics and Department of Physics, University of Michigan, Ann Arbor, MI, USA}
\email{schotland@umich.edu}

\begin{abstract}
The paper \cite{HosSch} describes a physical regime in which the ultrasound perturbation of a scattering optical medium leads to a frequency shift in some of the scattered light.  In this paper we consider the inverse problem of recovering the optical properties of this medium from measurements of the frequency-shifted light, using a radiative transport equation (RTE) model for light propagation.  Given some assumptions on the regularity and isotropicity of the coefficients of the RTE, we show that the absorption coefficient can be reconstructed from the boundary measurements of a single well chosen illumination, and that the scattering coefficients can be reconstructed from boundary measurements of a one-parameter family of illuminations.
\end{abstract}

\maketitle

\section{Introduction}

The acousto-optic effect is a phenomenon which occurs when the optical properties of a medium are perturbed by an acoustic wave. For example, given a compressible fluid whose optical properties depend on its density, an acoustic pressure wave which modulates the density of the fluid will also modulate its optical properties.  

The idea of acousto-optic imaging is to take advantage of this phenomenon to obtain a well-posed inverse problem leading to better reconstructions of interior data than can be obtained with solely acoustic or solely optical imaging.  Several mathematical models for acousto-optic imaging have been well studied; see ~\cite{ABGNS, AmmBosGarSep, AmmGarNguSep, AmmNguSep, BalChuSch, BalMos, BalSch, BalSch2, ChuSch} for examples.

However, the mathematical details of the acousto-optic inverse problem vary considerably depending on the properties of the optical medium and the manner in which it responds to acoustic waves. In this paper we will consider the acousto-optic inverse problem in the multifrequency regime described by Hoskins and Schotland in ~\cite{HosSch}, where the dielectric permittivity of the medium is perturbed by the acoustic wave.  Here the perturbation of the medium by acoustic waves leads to a detectable frequency shift in the scattered light. This enables an observer to illuminate the medium with a single source frequency and observe the scattered, frequency-shifted light separately.  

To describe the system more precisely, consider a bounded smooth domain $X \subset \R^3$, and suppose that the specific intensity of the source frequency is represented by the function $u: X \times S^2 \rightarrow \R$.  Here $u(x,\theta)$ represents the intensity of light at the point $x \in X$ in the direction $\theta \in S^{2}$.  Following ~\cite{HosSch} we model light propagation by a radiative transport equation (RTE), so $u$ satisfies the equation 
\begin{equation}\label{SourceRTE}
\theta \cdot \grad u(x,\theta) = -\sigma(x) u(x,\theta) + \int_{S^2} k(x,\theta,\theta')u(x,\theta') d\theta '  \mbox{ on } X \times S^2,
\end{equation}
with a prescribed boundary condition 
\[
u|_{\Gamma_{-}} = f
\]
on the incoming boundary $\Gamma_-$ defined by
\[
\Gamma_{\pm} = \{ (x,\theta) \in \partial X \times S^2 | \pm \theta \cdot \nu(x) > 0\}.
\]
For convenience, we will sometimes represent the operator on the right side of \eqref{SourceRTE} by $Au(x,\theta)$, so \eqref{SourceRTE} becomes
\begin{equation}\label{ShortRTE}
\theta \cdot \grad u = Au.
\end{equation}

Now consider a perturbation of the domain by an ultrasound wave of the form $\cos(Q \cdot x)$, for some $Q \in \R^3$. In ~\cite{HosSch}, the authors show that in the physical regime they describe, the ultrasound perturbation generates frequency-shifted light modeled by the equations
\begin{equation}\label{MFAOSystem}
\begin{split}
\theta \cdot \grad u_{00} &= Au_{00} \\
\theta \cdot \grad u_{01} &= Au_{01} + a \cos(Q\cdot x) u_{00} \\
\theta \cdot \grad u_{11} &= Au_{11} + b \cos(Q\cdot x) u_{01}, \\
\end{split}
\end{equation}
with boundary conditions $u_{00}|_{\Gamma_-} = f$, and $u_{01}|_{\Gamma_-} = u_{11}|_{\Gamma_-} = 0$.
Here $u_{00}: X \times S^2 \rightarrow \R$ represents the original source frequency, and $u_{11}, u_{01}:X \times S^2 \rightarrow \R$  represent the frequency-shifted light and its coherence with the original source frequency, respectively. 

We want to take advantage of this phenomenon to help reconstruct the optical coefficients $\sigma$ and $k$.  This gives rise to the following question: given boundary measurements of $u_{00}, u_{01},$ and $u_{11}$ for various $f$ and $Q$, can we reconstruct $\sigma$ and $k$?  (The coefficients $a$ and $b$, which govern the strength of the acousto-optic effect, are assumed to be known.) Note that the frequency-shifted light $u_{11}$ is doubly modulated by the ultrasound perturbation -- it takes a modulation of $u_{01}$ as its source, but $u_{01}$ is itself modulated by the ultrasound perturbation.  Therefore it makes sense to concentrate on boundary measurements of $u_{01}$, and ask if we can use these to reconstruct $\sigma$ and $k$. Note that a similar question is studied in \cite{ChuHosSch} in the case of a highly scattering regime, where light propagation is well-approximated by a diffusion equation.  In this paper, however, we analyze the full transport equation model for this question. 

To help formulate the question and its answer more precisely, we impose the following a priori conditions on the coefficients $\sigma$ and $k$.  

\begin{description}

\item[Regularity condition] 
\begin{equation}\label{RegularityCondition} 
\sigma \in C(X) \mbox{ and } k \in C(X \times S^2 \times S^2) \mbox{ are nonnegative. }
\end{equation}
\item[Absorption condition] Scattering does not generate light; in other words there exists $c > 0$ such that 
\begin{equation}\label{AbsorptionCondition} 
\inf_{x \in X} (\sigma - \rho) > c,
\end{equation}
where
\[
\rho(x) = \left\| \int_{S^{n-1}}  k(x,\theta, \theta') d\theta'  \right\|_{L^{\infty}(S^{n-1})}.
\]
\item[Isotropicity condition] Scattering is identical for incoming and outgoing light, so
\begin{equation}\label{Isotropik}
k(x, \theta, \theta ') = k(x, -\theta ', -\theta ).
\end{equation}

\end{description}

With these conditions on $\sigma$ and $k$, and an $L^{\infty}$ boundary source $f$, the equation \eqref{MFAOSystem} has unique $L^{\infty}$ solutions $u_{00}, u_{01},$ and $u_{11}$ in $X \times S^2$ ~\cite{ChuSch, DauLio} (see also Proposition \ref{CollisionExp} below).  

Therefore for each pair $\sigma, k$ satisfying satisfying the above conditions, we will define the boundary value map $\mathcal{A}^{01}_{\sigma, k}: \R^3 \times L^{\infty}(\Gamma_-) \rightarrow L^{\infty}(\Gamma_+)$ by
\[
\mathcal{A}^{01}_{\sigma, k}(Q,f) = u_{01}|_{\Gamma_+}.
\]
We are now ready to state the main result.

\begin{theorem}\label{MainTheorem}
Given $\sigma$ and $k$ satisfying \eqref{RegularityCondition}, \eqref{AbsorptionCondition}, and \eqref{Isotropik}, the map $(\sigma, k) \mapsto \mathcal{A}^{01}_{\sigma, k}$ is injective. 

In fact, there exists $f \in L^{\infty}(\Gamma_-)$ such that $\mathcal{A}^{01}_{\sigma,k}(Q,f)$ suffices to recover $\sigma$.  Morover, there is a one-parameter subset of $L^{\infty}(\Gamma_-)$ such that if we restrict the domain of $\mathcal{A}^{01}_{\sigma, k}$ to this subset, the map from $(\sigma,k)$ to the restricted map  $\mathcal{A}^{01}_{\sigma, k}$ is still injective.  
\end{theorem}

In other words, only one boundary source is needed to recover $\sigma$ and only a one parameter set of sources is required to reconstruct $k$.  

Three remarks should be made here.  First, the proof of Theorem \ref{MainTheorem} is constructive -- we will provide an explicit method of reconstructing $\sigma$ and $k$ from $\mathcal{A}^{01}$. Secondly, this construction comes with stability estimates: see Proposition \ref{FunctionalStability} and Theorem \ref{StabilityEstimate}.  Finally, the fact that we can rely on a one-parameter set of sources only provides an advantage over the ordinary optical tomography (without ultrasound) results of \cite{ChoSte}.

The proof of Theorem \ref{MainTheorem} goes loosely as follows.  First we use the measurements of $u_{01}$, together with an integration by parts, to obtain an internal functional (Section 2).  Then we consider the forward problem for the RTE (Section 3) and use the form of the solutions to analyze the functional.  In Section 4 we give an informal description of the method of proof for Theorem \ref{MainTheorem}, and in Section 5 we present the full proof, along with stability estimates.

\section{Internal Functional}

To start, we need an internal functional. Suppose $u_{01}$ is as above, and $v(x,\theta)$ solves the adjoint equation
\begin{equation}\label{AdjointRTE}
-\theta \cdot \grad v = Av,
\end{equation}
with the natural boundary condition $v|_{\Gamma_+} = g$ specified by us.  (Note that solutions to the adjoint RTE \eqref{AdjointRTE} are precisely solutions to the regular RTE \eqref{ShortRTE} under the change of variables $\theta \mapsto -\theta$). Integrating by parts,
\[
\int_X \theta \cdot \grad u_{01} v \, dx = -\int_X  u_{01} \theta \cdot \grad v \, dx + \int_{\partial X} u_{01} v \, \theta \cdot n \, dS,
\]
so
\[
\int_X (A u_{01} + a\cos(Q\cdot x)u_{00}) v \, dx = \int_X u_{01} A v \, dx + \int_{\partial X} u_{01} v \, \theta \cdot n \, dS.
\]
If we integrate in the $\theta$ variables also, then the isotropicity assumption in \eqref{Isotropik} guarantees that $A$ is self adjoint, and so
\[
\int_X \int_{S^2} a\cos(Q\cdot x)u_{00} v \, d\theta \, dx =  \int_{\partial X} \int_{S^2} u_{01} v \, \theta \cdot n \, d\theta \, dS.
\]
Since $u_{01}|_{\Gamma_-} = 0$, the right hand side reduces to an integral over $\Gamma_+$, so 
\begin{equation}\label{BoundaryIdentity}
\int_X \int_{S^2} a\cos(Q\cdot x)u_{00} v \, d\theta \, dx = \int_{\Gamma_+} u_{01} g \, \theta \cdot n \, d\theta \, dS.
\end{equation}
Here the right side can be measured, so the left hand side is also known.  If we vary the ultrasound modulation (and assume the amplitude $a$ is known), we recover the Fourier transform of the quantity 
\[
H(x) = \int_{S^2} u_{00} v \, d\theta.
\]
The challenge is to use this functional to recover $\sigma$ and $k$. Since the functional we measure depends on the boundary values we choose for $u_{00}$ and $v$, we can write
\begin{equation}\label{TheFunctional}
H(x) = H_{f,g}(x) = \int_{S^2} u_{00} v \, d\theta,
\end{equation}
where $f$ and $g$ are understood to be the boundary values of $u_{00}$ and $v$ respectively.

The recovery of $H$ from the boundary values comes with the following stability estimate. 

\begin{prop}\label{FunctionalStability}
If $H_1$ and $H_2$ are functionals obtained from the same initial data $(f,g)$, but separate sets of coefficients $\sigma_1, k_1$ and $\sigma_2, k_2$, we have the stability estimate
\begin{equation}
\|H_1 - H_2\|_{L^{\infty}(X)} \lesssim \|g\|_{L^{\infty}(\Gamma_{+})}\|\mathcal{A}^{01}_{\sigma_1, k_1}(Q,f)-\mathcal{A}^{01}_{\sigma_2, k_2}(Q,f)\|_{L^1(\R^n \times \Gamma_{+})}.
\end{equation}
\end{prop}

\begin{proof}
Note that the quantity on the left side of \eqref{BoundaryIdentity} is $\hat{H}(Q)$, and the $u_{01}|_{\Gamma_+}$ that appears on the right side can be rewritten as $\mathcal{A}^{01}_{\sigma, k}(Q,f)$.  Therefore \eqref{BoundaryIdentity} tells us that 
\[
\hat{H}_1(Q) - \hat{H}_2(Q) = \int_{\Gamma_+} (\mathcal{A}^{01}_{\sigma_1, k_1}(Q,f)-\mathcal{A}^{01}_{\sigma_2, k_2}(Q,f)) g \, \theta \cdot n \, d\theta \, dS,
\]
and the stability estimate now follows from standard estimates.  
\end{proof}

\section{RTE Solutions}

To use this functional, we need to take advantage of the \emph{collision expansion} for solutions to the RTE.  In order to write this down, we first need to fix some terminology.

For $x,y \in \overline{X}$, let 
\[
\tau(x,y) = \int_0^{|x-y|}\sigma(x - s(\widehat{x-y}))ds.
\]
Roughly speaking $\tau(x,y)$ represents the optical distance from $x$ to $y$ in the presence of the absorption coefficient $\sigma$, without scattering. Note that $\tau(x,y) = \tau(y,x)$.  

Define $\gamma_{\pm}:X \times S^2 \rightarrow \Gamma_{\pm}$ by setting $\gamma_{\pm}(x,\theta)$ to be the (first) point in $\partial X$ obtained by travelling from $x$ in the $\pm \theta$ direction; we can think of this as being the projection of $x$ onto $\partial X$ in the direction $\pm \theta$. 

Now we can define $J$ to be the operator which solves the non-scattering RTE
\begin{eqnarray*}
\theta \cdot \grad u &=& -\sigma u \\
u|_{\Gamma_-} &=& f, \\
\end{eqnarray*}
and write $J$ explicitly in terms of $\tau$ and $\gamma_-$ as  
\begin{equation}\label{JExplicit}
Jf(x,\theta) = e^{-\tau(x,\gamma_-(x,\theta))}f(\gamma_-(x,\theta),\theta).
\end{equation}
Similarly, if we define $T^{-1}$ to be the operator which solves the nonscattering RTE
\begin{eqnarray*}
\theta \cdot \grad u &=& -\sigma u +S\\
u|_{\Gamma_-} &=& 0, \\
\end{eqnarray*}
then explicitly
\begin{equation}\label{TInverse}
T^{-1}S(x,\theta) = \int_0^{|x-\gamma_-(x,\theta)|}e^{-\tau(x,x-t\theta)}S(x - t\theta, \theta)dt.
\end{equation}
Finally, define $A_2$ to be the scattering operator
\[
A_2 w = \int_{S^2} k(x, \theta, \theta') w(x,\theta ') d \theta ',
\]
and
\begin{equation}\label{Kexplicit}
Kw = T^{-1} A_2 w 
\end{equation}

The main point of this section is to record the Neumann series solution for the RTE. 

\begin{prop}\label{CollisionExpansion}
Suppose $\sigma$ and $k$ satisfy the conditions laid out in Section 1. Then there exists $0 < C < 1$ such that
\[
\|K\|_{L^{\infty}(X \times {S^{n-1}}) \rightarrow L^{\infty}(X \times {S^{n-1}})} < C, 
\]
and if $u$ solves the RTE $\theta \cdot \grad u = Au$ with the boundary condition $u|_{\Gamma_{-}} = f$, for some $f \in L^{\infty}(\Gamma_{-})$, then $u$ takes the form
\begin{equation}\label{CollisionExp}
u = (1 + K + K^2 + \ldots)Jf,
\end{equation}
\end{prop}

See e.g. ~\cite{BalChuSch, ChoSte, DauLio} for proofs. The expansion \eqref{CollisionExp} is the collision expansion of $u$, and it is significant mainly because $K$ is a smoothing operator, so each subsequent term of the expansion is less singular.  The primary term $Jf$ corresponds to light propagation in the absence of scattering, and is called the ballistic term.  Loosely speaking, the $K^mJf$ term corresponds to light that has been scattered $m$ times, and so $KJf$ can be referred to as the single-scattering term, $K^2Jf$ as the double-scattering term, and so on.

Note that analogous results also hold for the adjoint equation \eqref{AdjointRTE}, with appropriate corresponding operators $K^*$, $J^*$, etc. obtained via the change of variables $\theta \mapsto -\theta$. 

The following estimates, taken from Chung-Schotland ~\cite{ChuSch} (see the proof of Corollary 3.2) will also prove useful later.

\begin{lemma}\label{KEstimates}
Note that at any $x \in X$,
\begin{equation}\label{A2wLinfty}
\|A_2(w)(x, \cdot)\|_{L^{\infty}({S^{n-1}})} < C_k\|w(x, \cdot)\|_{L^1({S^{n-1}})}.
\end{equation}
Moreover
\begin{equation}\label{T1Estimate}
\|T_1^{-1}w\|_{L^{\infty}(X \times {S^{n-1}})} < \|w\|_{L^{\infty}(X \times {S^{n-1}})}.  
\end{equation}
\end{lemma}

\section{Point-Plane Inversion}

The main difficulty in obtaining $\sigma$ and $k$ from the functional $H_{f,g}$ is the nonlinearity of the functional. The basic idea for countering this difficulty is to use Proposition \ref{CollisionExpansion} with carefully chosen boundary conditions $f$ and $g$ to ensure that only the leading order terms contribute meaningfully to $H_{f,g}$.  

This is similar in concept to the idea used in ~\cite{ChuSch}, but with the important difference that in our case, the principal term in the expansion carries no information. This is best seen by understanding what happens in the absence of scattering. In that case, the operator $K$ becomes zero and the solutions to the RTE are given entirely by the ballistic term. But now the quantity $u_{00}v$ satisfies the equation
\[
\theta \cdot \grad (u_{00}v) = 0
\]
so this quantity does not vary as we move into the domain from the boundary.  Therefore our functional is useless in the absence of scattering. It follows that we want to draw information not from the leading order term in the collision expansion, but from the subsequent term. 

In this section we will give an informal discussion on how this can be done -- first by considering each point $x$ in the domain one at a time, and then foliating the domain with planes and considering each plane one at a time.  In the following section, we'll describe how this process can be extended to consider the entire domain at once, and we'll also make the discussion fully rigorous in the process.

\subsection{Point Sources}
We'll begin by considering one point at a time.  We want to fix an $x \in X$ and consider point sources on the boundary aimed in the direction of $x$.
To do this, define for a pair $(x_0, \theta_0) \in \partial X \times S^2$ the delta distribution $\delta_{x_0, \theta_0}$ so that 
\[
\int_{\partial X \times S^2} \delta_{x_0, \theta_0} f = f(x_0, \theta_0)
\]
for any $f \in C^{\infty}(\partial X \times S^2)$. Now consider a solution $u$ to the RTE with boundary data given by such a delta function. (Making this idea rigorous requires some redefinition of the notion of solution to encompass distributions, which we do not address here. The discussion in the next section will contain a rigorous analysis in terms of approximations to delta distributions, which makes more sense in the context of implementation.)  By Proposition \ref{CollisionExpansion}, 
\begin{equation}\label{SingularCollisionExpansion}
u = J\delta_{x_0, \theta_0} + KJ\delta_{x_0, \theta_0} + K^2 J\delta_{x_0, \theta_0} + \ldots 
\end{equation}
Here $J\delta_{x_0, \theta_0}$ is a distribution supported on the codimension four subset of the (five dimensional set) $X \times S^2$ given by
\[
\{(x, \theta_0)| x = x_0 + c \theta_0 \mbox{ for some } c \in \R\}
\]
The operator $K$ integrates in one spatial dimension and two angular dimensions, so $KJ\delta_{x_1, \theta_1}$ is supported on a codimension one subset, and all subsequent terms are less singular.   

Now fix an $x \in X$, and $\theta_1,\theta_2 \in S^{n-1}$ such that $\theta_1 \neq \theta_2$. We set $x_1 = \gamma_-(x, \theta_1)$ and $x_2 = \gamma_+(x,\theta_2)$ (see Figure 1.)

We define corresponding boundary sources $f = \delta_{x_1, \theta_1}$ and $g = \delta_{x_2, \theta_2}$, and consider the resulting functional 
\[
H_{f,g}(x) = \int_{S^2} u_{00}(x,\theta) v(x,\theta) d\theta.
\] 
The integral identity \eqref{BoundaryIdentity} implies that this corresponds to the boundary observation of $u_{01}$ at $x_2$ in the direction of $\theta_2$, where the ultrasound beam has been focused to concentrate its support at $x$.  

\begin{center}
\includegraphics[scale=0.5]{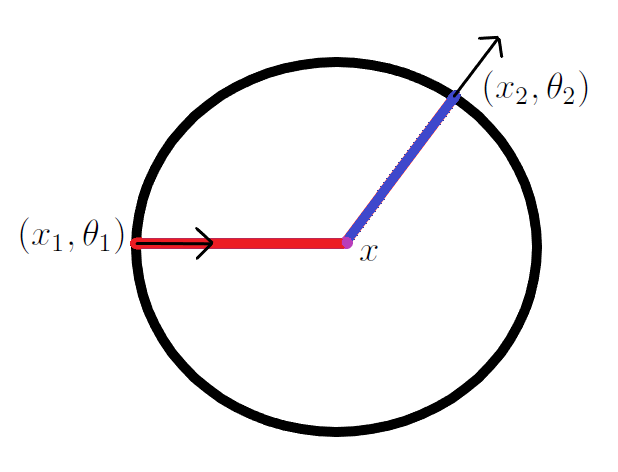}
\end{center}

\small Figure 1: $H_{f,g}(x)$ represents light from the point source $(x_1,\theta_1)$, which is scattered and 

frequency shifted from $x$, and observed at $(x_2,\theta_2)$.

\vspace{4mm}

Thus we could expect that the leading term in the functional will represent light travelling straight from $x_1$ to $x$, scattering once at $x$ in the direction of $\theta_2$, and exiting at $x_2$.  

Indeed, this is what we see when we expand $u_{00}$ and $v$ in terms of the collision expansion \eqref{SingularCollisionExpansion}.  Since $\theta_1 \neq \theta_2$, the leading term $JfJ^*g$ is zero, so the dominant terms of $H_{f,g}$ are 
\[
\int_{S^2} (Jf K^*J^*g + KJf J^*g) d\theta.
\]
Each of these terms represents a distribution supported on a codimension one set multiplied by one supported on a codimension four set.  Expanding out $Jf K^{*}J^{*}g$ at $x$ using \eqref{JExplicit} and \eqref{Kexplicit} gives
\[
\int_{S^2}Jf K^{*}J^{*}g \, d\theta = k(x, -\theta_1, -\theta_2)\exp(-\tau(x,x_1) - \tau(x, x_2))\delta_x(x),
\]
where $\delta_x(x)$ reflects the size of the distribution at $x$. Similarly, 
\[
\int_{S^2}KJf J^{*}g \, d\theta = k(x, \theta_2, \theta_1)\exp(-\tau(x,x_1) - \tau(x, x_2))\delta_x(x).
\]
By \eqref{Isotropik}, these terms are identical, so to leading order, and ignoring the $\delta_x(x)$ factor, we get 
\begin{equation}\label{BrokenRay}
H_{f,g}(x) \simeq 2k(x, \theta_2, \theta_1)\exp(-\tau(x,x_1) - \tau(x,x_2)).
\end{equation}
In other words, $H_{f,g}$ gives the scattering factor at $x$ between the angles $\theta_1$ and $\theta_2$, times the total attenuation from $x_1$ to $x$ to $x_2$, which is exactly we would expect from the discussion above Figure 1.  

If $\tau(x,y)$ is known for all pairs $(x,y)$, then $k(x,\theta_2, \theta_1)$ can be read off directly from this formula. If not, then we can set $x_1 = \gamma_-(x,\theta_1)$ and $x_2 = \gamma_+(x,\theta_1)$, and measure two functionals
\begin{eqnarray*}
H_1 = H_{\delta_{x_1,\theta_1},\delta_{x_1, \theta_1}}(x) &\simeq& 2k(x, \theta_1, -\theta_1)\exp(-2\tau(x,x_1)) \\
H_2 = H_{\delta_{x_2,-\theta_1},\delta_{x_2, -\theta_1}}(x) &\simeq& 2k(x, \theta_1, -\theta_1)\exp(-2\tau(x,x_2)), \\
\end{eqnarray*}
and the additional quantity 
\[
H_3 = \exp(-\tau(x_1,x_2))
\]
which can be obtained from the albedo map $\mathcal{A}^{00}$ for $u_{00}$, applied to the point source $\delta_{x_1,\theta_1}$. 

By the additivity of $\tau$, we have $\tau(x_1,x_2) = \tau(x,x_1) + \tau(x,x_2)$, so we get 
\[
\tau(x,x_1) = \half(\log H_1 - \log H_2 + \log H_3).
\]
Repeating the exercise for different $x$, and using the additivity of $\tau$, if necessary, gives any desired value of $\tau(x,y)$.

Differentiating $\tau$ gives $\sigma(x)$, so this discussion tells us that we can recover $\sigma$ and $k$ completely from the functional $H$.  On the other hand, using the methods described above mean that in order to completely recover $\sigma$ and $k$, we need to consider all possible point sources, which means we need four dimensions worth of sources.  We can improve this slightly by taking plane sources instead.

\subsection{Plane Sources}

Fix a $\theta_0 \in S^{2}$, and fix a plane $P$ parallel to $\theta_0$ which intersects the set $\{x \in \partial X| (x, \theta_0) \in \Gamma_-\}$.  Let $\delta_{P,\theta_0}$ be a distribution supported on the set $P' = \{(x,\theta) \in \Gamma_-\| x \in P, \theta = \theta_0\}$, so that 
\[
\int_{\partial X \times S^2} \delta_{P, \theta_0} f = \int_{P'} f(x, \theta_0) 
\]
for any $f \in C^{\infty}(\partial X \times S^2)$. In other words, $\delta_{P, \theta_0}$ is a distribution supported on a codimension three subset of the four dimensional set $\partial X \times S^2$, and if we view $\delta_{P, \theta_0}$ as a boundary source for the RTE and write out the collision expansion
\begin{equation}\label{PlaneCollisionExpansion}
u = J\delta_{P, \theta_0} + KJ\delta_{P, \theta_0} + K^2 J\delta_{P, \theta_0} + \ldots, 
\end{equation}
the leading term $J\delta_{P, \theta_0}$ is a distribution supported on a codimension 3 subset of the (five dimensional) domain $X \times S^2$.  

Since $K$ integrates in one spatial dimension and two angular dimensions, $KJ\delta_{P, \theta_0}$ is supported everywhere.  However, it is not actually a function: consider 
\[
KJ\delta_{P, \theta_0}(x,\theta) = T^{-1}A_2 J\delta_{P, \theta_0}(x,\theta)
\] 
for $x \in P$ and $\theta$ parallel to $P$.  Then the spatial integral in $T^{-1}$ is an integral along a line fully contained in $P$, so this spatial integral does not reduce the singularity of the distribution $A_2 J\delta_{P, \theta_0}(x,\theta)$.  Therefore $KJ\delta_{P, \theta_0}(x,\theta)$ can be thought of as a function supported on $X \times S^2$ plus a distribution supported on the codimension one set $P \times S^2$.  

Now given the choice of $\theta_1 \in S^2$ and $P$ parallel to $\theta_1$, pick $(x_2,\theta_2) \in \Gamma_+$ so that $x_2$ lies in $P$ and $\theta_2$ is parallel to $P$, with $\theta_1 \neq \theta_2$.  

We define corresponding boundary sources $f = \delta_{P, \theta_1}$ and $g = \delta_{x_2, \theta_2}$, and consider the resulting functional $H_{f,g}(x)$ at any point $x$ in the line through $x_2$ in direction $-\theta_2$.  As in the point source case, the ballistic terms multiply to give zero. By the above discussion, the leading term for what remains is the 
\[
KJ\delta_{P, \theta_1}(x,\theta) J^{*}\delta_{x_2, \theta_2}
\]
term; this represents a codimension four distribution multiplied by a codimension one distribution.  Expanding using \eqref{JExplicit} and \eqref{Kexplicit} gives
\begin{equation}\label{PointPlane}
H_{f,g}(x) \simeq \int_0^{|x-\gamma_-(x,\theta_2)|}e^{-\tau(\gamma_+(x,\theta_2),x-t\theta_2)-\tau(x-t\theta_2,\gamma_-(x-t\theta_2,\theta_1))}k(x-t\theta_2,\theta_2,\theta_1)dt \delta_x(x)
\end{equation}

\begin{center}
\includegraphics[scale=0.5]{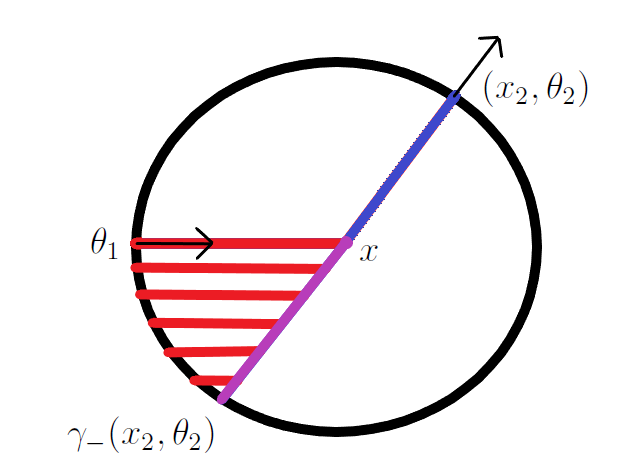}
\end{center}

{\small Figure 2: $H_{f,g}(x)$ represents the light from the plane source $\delta_{P,\theta_1}x_1$, scattered and 

frequency-shifted along the line from $\gamma_-(x,\theta_2)$ to $x$, and thence transmitted to $(x_2,\theta_2)$.  }

\vspace{4mm}

Ignoring the $\delta_x(x)$ factor and taking the directional derivative in direction $\theta_2$, we get 
\[
\theta_2 \cdot \grad H_{f,g}(x) \simeq k(x, \theta_2, \theta_1)\exp(-\tau(x,x_1)-\tau(x,x_2)).
\]
which is just \eqref{BrokenRay}, and so the remainder of the reconstruction proceeds as in the point source case.

Note that for each plane source $\delta_{P, \theta_1}$, we can, by varying $x_2$ and $\theta_2$, recover a two dimensional collection of $k(x, \theta_1, \theta_2)$.  Therefore in this scenario only two dimensions worth of sources are needed to recover all of $k$ and $\sigma$.  

In fact it's possible to do better than this: we can restrict ourselves to using a single dimension worth of sources, if we use an angularly singular source like $\delta_{\theta_1}$ and multiply by a rapidly oscillating function.  This brings us to the proof of Theorem \ref{MainTheorem}.

\section{Reconstruction and Stability}

\subsection{Proof of Theorem \ref{MainTheorem}}
We begin by defining the following $L^{\infty}$ approximation to the delta function on $S^2$. 
\[
\delta^h_{\theta_1}(\theta) = \left\{ \begin{array}{ll} h^{-2} & \mbox{ if } |\theta - \theta_1| < h \\ 
                                                        0      & \mbox{ otherwise. }  \end{array} \right.
\]
Following the discussion at the end of the previous section, we want to multiply by a function that oscillates rapidly in the spatial directions perpendicular to $\theta_1$.  To do this, fix $\theta_1 \in S^2$, and $\theta_3$ perpendicular to $\theta_1$. Pick coordinates for $x$ such that $\theta_1 = \hat{x}_1$ and $\theta_3 = \hat{x}_3$.

Let 
\begin{equation}\label{OscillationSource}
f_h^{\theta_1}(x,\theta) = \delta^h_{\theta_1}(\theta)\exp(ix_3/h).
\end{equation}
This complex source is not physical, but it can be recreated formally by measuring the real and imaginary parts separately. (The advantage of using the complex exponential instead of a simple sine or cosine function is mostly formal anyway -- it prevents the boundary source from going to zero.) Using the collision expansion, we claim the following qualitative properties for the solutions of the RTE with boundary source $f$. 

\begin{lemma}\label{PlaneOscillation}
Suppose $f = f_h^{\theta_1}$ is as defined in \eqref{OscillationSource}, and $u$ is the solution to the RTE \eqref{ShortRTE} with boundary condition $u|_{\Gamma_-} = f$.  Then $u = Jf + KJf + R$, where 

\begin{itemize}
\item The ballistic term $Jf$ satisfies the estimates
\[
\|Jf\|_{L^{\infty}(X \times S^2)} = O(h^{-2}) \mbox{ and for any fixed } x, \quad  \|Jf(x,\cdot)\|_{L^{1}(S^2)} = O(1);
\]

\item The single scattering term $KJg$ satisfies the estimates
\[
\|KJf\|_{L^{\infty}(X \times S^2)} = O(1) \mbox{ and for any fixed } x, \quad  \|KJf(x,\cdot)\|_{L^{1}(S^2)} = o(1);
\]

\item and the remainder satisfies the estimate
\[
\|R\|_{L^{\infty}(X \times S^2)} = o(1).
\]
\end{itemize}

\end{lemma} 

\begin{proof}
The estimates for $Jf$ follow straight from the definitions of $J$ and $f$.  Then the $L^{\infty}$ norm for $KJf$ follows from Lemma \ref{KEstimates} and the $L^1$ estimate for $Jf$.  

Now 
\[
Jf(x,\theta) = e^{-\tau(x,\gamma_-(x,\theta))}\delta^h_{\theta_1}(\theta)\exp(i\hat{x}_3 \cdot \gamma_-(x,\theta)/h).
\]
Therefore
\[
A_2 Jf(x,\theta) = \int_{S^2} k(x,\theta,\theta ')e^{-\tau(x,\gamma_-(x,\theta'))}\delta^h_{\theta_1}(\theta')\exp(i\hat{x}_3 \cdot \gamma_-(x,\theta')/h) \, d\theta'.
\]
Since $\delta^h_{\theta_1}$ is supported only for $\theta$ in a small neighbourhood of $\theta_1$, the Lebesgue differentiation theorem guarantees that for sufficiently small $h$, we get
\[
A_2 Jf(x,\theta) = e^{-\tau(x,\gamma_-(x,\theta_1))}k(x,\theta, \theta_1)\exp(i\hat{x}_3 \cdot \gamma_-(x,\theta_1)/h) + o(1).
\]
Since $\theta_1$ is perpendicular to $\hat{x}_3$, we get 
\[
A_2 Jf(x,\theta) = e^{-\tau(x,\gamma_-(x,\theta_1))}k(x,\theta, \theta_1)\exp(ix_3/h) + o(1).
\]
Now we can write $KJf$ as 
\[
T^{-1}A_2Jf(x,\theta) = \int_0^{|x-\gamma_-(x,\theta)|}e^{-\tau(x,x-t\theta)}A_2Jf(x - t\theta, \theta)dt.
\]
If $\theta \cdot \hat{x}_3 \gg h$, then $A_2Jf(x - t\theta, \theta)$ is highly oscillatory as a function of $t$, and so by the Riemann-Lebesgue lemma,
\[
|KJf(x,\theta)| = o(1).
\]
Then it follows that 
\[
\|KJf(x, \cdot)\|_{L^1(S^2)} = o(1),
\]
and the estimate for $R$ follows from Lemma \ref{KEstimates}.

\end{proof}

We want to look at the functional $H_{f_h^{\theta_1},g_h^{\theta_2}}$ defined by $f_h^{\theta_1}$ and a boundary function $g_h^{\theta_2}$ which approximates a point source.  To define $g_h^{\theta_2}$,  we'll begin by defining the approximation to the delta function on the boundary: for $x_0 \in \partial X$, define 
\[
\delta^h_{x_0}(x) = \left\{ \begin{array}{ll} h^{-2} & \mbox{ if } |x - x_0| < h \\ 
                                              0      & \mbox{ otherwise }  \end{array} \right.
\]
Pick $\theta_2 \in S^2$ so $\theta_2$ is perpendicular to $\hat{x}_3$, and let 
\begin{equation}\label{ThePointSource}
g_h^{\theta_2}(x,\theta) = h^2 \delta^h_{\theta_2}(\theta)\delta^h_{x_0}(x).  
\end{equation}

\begin{lemma}\label{PointEstimates}
Let $g = g_h^{\theta_2}$ be defined by \eqref{ThePointSource}, and let $v$ solve the adjoint RTE \eqref{AdjointRTE} with boundary condition $v|_{\Gamma_+} = g|_{\Gamma_+}$. Then
\[
v = J^{*}g + K^*J^*g + R^*,
\]
where

\begin{itemize}
\item The ballistic term $J^*g$ satisfies the estimates
\[
\|J^*g\|_{L^{\infty}(X \times S^2)} = O(h^{-2}) \mbox{ and for any fixed } x, \quad  \|J^*g(x,\cdot)\|_{L^{1}(S^2)} = O(1);
\]

\item The single scattering term $K^*J^*g$ satisfies the estimates
\[
\|K^*J^*g\|_{L^{\infty}(X \times S^2)} = O(1) \mbox{ and for any fixed } x, \quad  \|K^*J^*g(x,\cdot)\|_{L^{1}(S^2)} = o(1);
\]

\item and the remainder satisfies the estimate
\[
\|R^*\|_{L^{\infty}(X \times S^2)} = o(1).
\]
\end{itemize}

\end{lemma}

\begin{proof}
The estimates for $J^*g$ and the $L^{\infty}$ estimate for $K^*J^*g$ are obtained in the same manner as in Lemma \ref{PlaneOscillation}.  To get the $L^1$ estimate for $K^*J^*g$, note that $J^*g(x,\theta)$ is only supported for $x$ within $O(h)$ distance of the line from $x_0$ in direction $\theta_2$.  Therefore $A_2^*J^*g(x,\theta)$ is only supported for $x$ within $O(h)$ distance of this line. Then for $\theta$ such that $|\theta- \theta_2| \gg h$, 
\[
K^*J^*g(x,\theta) = T^{*-1}A_2^*J^*g(x,\theta) = \int_0^{|x-\gamma_+(x,\theta)|}e^{-\tau(x,x+t\theta)}A_2^*J^*f(x + t\theta, \theta)dt
\]
and the integrand is supported only in an $O(h)$ segment of the line.  Therefore
\begin{equation}\label{LocalSecondTermEstimate}
K^*J^*g(x,\theta) = O(h)
\end{equation}
for $|\theta- \theta_2| \gg h$, and the $L^1$ estimate for $K^*J^*g$ follows.  

The estimate for $R^*$ now follows from Lemma \ref{KEstimates}.
\end{proof}

Now let's consider the functional $H_{f,g}$ obtained from the sources $f$ and $g$ described above.  Using Lemmas \ref{PlaneOscillation} and \ref{PointEstimates} respectively, we can expand the functional as

\begin{eqnarray*}
H_{f,g} &=& \int_{S^2} Jf J^*g \, d\theta + \int_{S^2} KJf J^*g \, d\theta + \int_{S^2} R J^*g \, d\theta \\
	      & & \int_{S^2} Jf K^*J^*g \, d\theta + \int_{S^2} KJf K^*J^*g \, d\theta + \int_{S^2} R K^*J^*g \, d\theta \\
        & & \int_{S^2} Jf R^* \, d\theta + \int_{S^2} KJf R^* \, d\theta + \int_{S^2} R R^* \, d\theta. \\
\end{eqnarray*}

Assuming that $|\theta_1 - \theta_2| \gg h$, the first term consists of two functions angularly supported on disjoint subsets of $S^2$, so it vanishes. Moreover, applying Lemmas \ref{PlaneOscillation} and \ref{PointEstimates} shows that six of the remaining terms are $o(1)$ at best. What remains is 
\[
H_{f,g} = \int_{S^2} KJf J^*g \, d\theta + \int_{S^2} Jf K^* J^* g \, d\theta + o(1).  
\]
But the $\int_{S^2} Jf K^* J^* g \, d\theta$ term isn't quite as big as advertised.  Assuming that $|\theta_1 - \theta_2| \gg h$, we have from \eqref{LocalSecondTermEstimate} that $K^* J^* g(x,\theta) = O(h)$ for $\theta$ in the support of $Jf$.  Therefore this term is $O(h)$, and as a result we are left with
\[
H_{f,g} = \int_{S^2} KJf J^*g \, d\theta + o(1).  
\]
Now 
\[
J^*g(x,\theta) = e^{-\tau(x,\gamma_+(x,\theta))} h^2\delta^h_{\theta_2}(\theta)\delta^h_{x_0}(\gamma_+(x,\theta)).
\]
Therefore
\[
H_{f,g}(x) = e^{-\tau(x,\gamma_+(x,\theta_2))} h^2 \delta^h_{x_0}(\gamma_+(x,\theta_2)) KJf(x,\theta_2) + o(1).  
\]
For $x$ such that $\gamma_+(x,\theta_2)$ is in the support of $\delta^h_{x_0}$, we can write
\begin{equation}\label{H5}
H_{f,g}(x) = e^{-\tau(x,\gamma_+(x,\theta_2))} KJf(x,\theta_2) + o(1).  
\end{equation}
Meanwhile
\[
Jf(x,\theta) = e^{-\tau(x,\gamma_-(x,\theta))}\delta^h_{\theta_1}(\theta)\exp(i\hat{x}_3 \cdot \gamma_-(x,\theta)/h).
\]
so integrating against the scattering kernel gives
\[
A_2Jf(x, \theta_2) = e^{-\tau(x,\gamma_-(x,\theta_1))}\exp(i\hat{x}_3 \cdot \gamma_-(x,\theta_1)/h)k(x,\theta_2,\theta_1) + o(1).
\]
Since $\theta_1$ is perpendicular to $\hat{x}_3$, 
\[
A_2Jf(x, \theta_2) = e^{-\tau(x,\gamma_-(x,\theta_1))}\exp(ix_3/h)k(x,\theta_2,\theta_1) + o(1).
\]
Now $K = T^{-1}A_2$, so 
\[
KJf(x,\theta_2) = \int_0^{|x-\gamma_-(x,\theta_2)|}e^{-\tau(x,x-t\theta_2)}A_2Jf(x - t\theta_2, \theta_2)dt + o(1).
\]
Substituting this into \eqref{H5} gives
\begin{eqnarray*}
H_{f,g}(x) &=& e^{-\tau(x,\gamma_+(x,\theta_2))} \cdot \\
           & & \int_0^{|x-\gamma_-(x,\theta_2)|}e^{-\tau(x,x-t\theta_2)-\tau(x-t\theta_2,\gamma_-(x-t\theta_2,\theta_1))}e^{i\hat{x}_3 \cdot (x -t\theta_2)/h}k(x-t\theta_2,\theta_2,\theta_1)dt \\
				   & & + o(1). \\ 
\end{eqnarray*}
Since $\theta_2$ is also perpendicular to $\hat{x}_3$, we can rewrite $\exp(i\hat{x}_3 \cdot (x -t\theta_2)/h) = \exp(ix_3/h)$. In fact, since $x$ is known, $\exp(ix_3/h)$ is also known, and we may as well assume that this is $1$.  Then we can write 
\[
H_{f,g}(x) = e^{-\tau(x,\gamma_+(x,\theta_2))}\int_0^{|x-\gamma_-(x,\theta_2)|}e^{-\tau(x,x-t\theta_2)}e^{-\tau(x-t\theta_2,\gamma_-(x-t\theta_2,\theta_1))}k(x-t\theta_2,\theta_2,\theta_1)dt + o(1).  
\]
Combining the remaining exponentials, we get
\[
H_{f,g}(x) = \int_0^{|x-\gamma_-(x,\theta_2)|}e^{-\tau(\gamma_+(x,\theta_2),x-t\theta_2)-\tau(x-t\theta_2,\gamma_-(x-t\theta_2,\theta_1))}k(x-t\theta_2,\theta_2,\theta_1)dt + o(1).  
\]
Up to the $o(1)$ error, note that this is precisely equation \eqref{PointPlane}, and has the same interpretation in terms of Figure 2.  

If we now consider $H_{f,g}(x - s\theta_2)$, for some parameter $s$, then using the expression above, we can write $H_{f,g}(x - s\theta_2)$ as
\[
\int_0^{|x-\gamma_-(x,\theta_2)|-s}e^{-\tau(\gamma_+(x,\theta_2),x-(t+s)\theta_2)-\tau(x-(t+s)\theta_2,\gamma_-(x-t\theta_2,\theta_1))}k(x-(t+s)\theta_2,\theta_2,\theta_1)dt + o(1).  
\]
Changing variables, we get
\[
\int_s^{|x-\gamma_-(x,\theta_2)|}e^{-\tau(\gamma_+(x,\theta_2),x-t\theta_2)-\tau(x-t\theta_2,\gamma_-(x-t\theta_2,\theta_1))}k(x-t\theta_2,\theta_2,\theta_1)dt + o(1).  
\]
If we take a difference quotient with respect to $s$, we get 
\begin{eqnarray*}
&  & \frac{H_{f,g}(x) - H_{f,g}(x - s\theta_2)}{s}  \\
&=& \frac{1}{s}\int_{0}^{s}e^{-\tau(\gamma_+(x,\theta_2),x-t\theta_2)-\tau(x-t\theta_2,\gamma_-(x-t\theta_2,\theta_1))}k(x-t\theta_2,\theta_2,\theta_1)dt + \frac{o_h(1)}{s}. \\
\end{eqnarray*}
Here we are emphasizing that the $o(1)$ term at the end is $o(1)$ as $h \rightarrow 0$.  If we take $0 < h \ll s \ll 1$ small, we get 
\begin{equation}\label{BrokenRayII}
\theta_2 \cdot \grad H_{f,g}(x) = e^{-\tau(\gamma_+(x,\theta_2),x)-\tau(x,\gamma_-(x,\theta_1))}k(x,\theta_2,\theta_1) + o_s(1).
\end{equation}
This is exactly the same quantity that we recovered in \eqref{BrokenRay} in the point source case, and the rest of the recovery proceeds exactly as in Section 4.1.  It helps to introduce the notation
\begin{eqnarray*}
F(x,\theta_1,\theta_2) &=& e^{-\tau(\gamma_+(x,\theta_2),x)-\tau(x,\gamma_-(x,\theta_1))}k(x,\theta_2,\theta_1) \\
                       &=& \theta_2 \cdot \grad H_{f^{\theta_1}_h,g^{\theta_2}_h}(x)+ o(1).\\
\end{eqnarray*}
to express equation \eqref{BrokenRayII}. Then explicitly, the discussion at the end of Section 4.1 implies that
\begin{equation}\label{TauReconstruction}
\tau(x,\gamma_-(x,\theta_1)) = \half(\log F(x, \theta_1,-\theta_1) -\log F(x, -\theta_1,\theta_1) +\log \mathcal{A}_{00}(f)(\gamma_+(x,\theta_1))),
\end{equation}
and
\begin{equation}\label{KReconstruction}
k(x,\theta_2,\theta_1) = F(x,\theta_1,\theta_2)e^{+\tau(\gamma_+(x,\theta_2),x)+\tau(x,\gamma_-(x,\theta_1))}.
\end{equation}

Note that if $\theta_1$ is fixed, then for a single boundary source parametrized by a choice of $\hat{x}_3$, we can, by changing $v$, obtain $k(x,\theta_2,\theta_1)$ for all $x$ and all $\theta_2$ perpendicular to $\hat{x}_3$.  By rotating the choice of $\hat{x}_3$, we can then obtain $k(x,\theta_2,\theta_1)$ for all $x$ and $\theta_2$.  Then \eqref{Isotropik} guarantees that we recover all $k(x,\theta_1, \theta_2)$.  This finishes the proof of Theorem \ref{MainTheorem}.

\subsection{Stability Estimates}
Equations \eqref{TauReconstruction} and \eqref{KReconstruction}, combined with \eqref{BrokenRayII}, immediately give us the following stability estimates.

\begin{theorem}\label{StabilityEstimate}
Suppose $\sigma_1, k_1$, and $\sigma_2,k_2$ are two sets of coefficients giving rise to two functionals $H_1$ and $H_2$.  Then
\[
\|\sigma_1 - \sigma_2\|_{C(X)} \leq \half\|\log |\grad H_1| - \log |\grad H_2| \|_{C^1(X)}
\]
and
\[
\|k_1 - k_2\|_{C(X \times S^2 \times S^2} \leq \sup_{x,y \in X}\exp(2\tau(x,y))\|H_1 - H_2\|_{C^1(X)}.
\]
\end{theorem}


\end{document}